\documentclass{article}
\usepackage[hidelinks]{hyperref}
\usepackage{amsmath}
\usepackage{amsfonts}
\usepackage{amssymb}
\usepackage{amsthm}
\usepackage{ltablex}
\usepackage{graphicx}
\usepackage{tikz}
\usetikzlibrary{shapes}

\newtheorem{theorem}{Theorem}
\newtheorem{corollary}[theorem]{Corollary}
\newtheorem{definition}[theorem]{Definition}

\newtheorem{lemma}[theorem]{Lemma}
\newtheorem{proposition}[theorem]{Proposition}

\title{Defensive alliance polynomial}
\author{Hany Ibrahim
	\\University of Applied Sciences Mittweida
}

\begin{document}
\maketitle

\begin{abstract}
We introduce a new bivariate polynomial which we call the \emph{defensive alliance polynomial} and denote it by \emph{$da(G;x,y)$}. It is a generalization of the alliance polynomial \cite{Carb14} and the strong alliance polynomial \cite{Carb16}. We show the relation between $da(G;x,y)$ and the alliance, the strong alliance and the induced connected subgraph \cite{Titt11} polynomials. Then, we investigate information encoded in $da(G;x,y)$ about $G$. We discuss the defensive alliance polynomial for the path graphs, the cycle graphs, the star graphs, the double star graphs, the complete graphs, the complete bipartite graphs, the regular graphs, the wheel graphs, the open wheel graphs, the friendship graphs, the triangular book graphs and the quadrilateral book graphs. Also, we prove that the above classes of graphs are characterized by its defensive alliance polynomial. A relation between induced subgraphs with order three and both subgraphs with order three and size three and two respectively, is proved to characterize the complete bipartite graphs. Finally, we present the defensive alliance polynomial of the graph formed by attaching a vertex to a complete graph. We show two pairs of graphs which are not characterized by the alliance polynomial but characterized by the defensive alliance polynomial.
\end{abstract}

\section{Introduction}
Let $G$ be a simple graph and $S$ be a subset of $V(G)$. $\bar{S}$ is $V(G)\setminus S$. The degree of a vertex $u$ in $S$ denoted by $\delta_{S}(u)$ is $|\{\{u,v\} \in E(G):v \in S\}|$.
An \emph{alliance} is a non-empty subset of $V(G)$. $S$ is \emph{defensive alliance} \cite{Kris02} provided that
\[
\delta_{S}(v) - \delta_{\bar{S}}(v) \geq -1 \; ,\; \forall \; v \in S \text{.}
\]
Further, $S$ is called \index{strong defensive alliance}\emph{strong defensive alliance} provided that:
\[
\delta_{S}(v) - \delta_{\bar{S}}(v) \geq 0 \; , \; \forall \; v \in S \text{.}
\]
The concept can be generalized to the \index{defensive $k$-alliance} \emph{defensive $k$-alliance} \cite{Rodr08}
\[
\delta_{S}(v) - \delta_{\bar{S}} \geq k \; , \forall \; v \in S \; , \;k \text{ is an integer in the range } -\Delta \leq k \leq \Delta \text{.}
\]
Note that for $k=-1$ we get the defensive alliance and for $k=0$ we get the strong defensive alliance. 

We denote by $G[S]$, the subgraph induced by $S$ in the graph $G$, where $S \subseteq V(G)$.
Through this paper, we present the graph polynomials using the form:
\[
\sum_{S \subseteq V(G) }[\, p_{1}(S) \, ][ \, p_{2}(S) \, ] \cdots x^{f_{x}(S)} y^{f_{y}(S)} z^{f_{z}(S)} \cdots \text{, where}
\]

\begin{equation*}
[\, p_{i}(S) \, ] =
\left\{
\begin{array}{ll}
1  & \mbox{if } G[S] \text{ has the property }p_{i} \text{,}\\
0  & \mbox{otherwise.} 
\end{array}
\right.
\end{equation*}
We denote the polynomial of the terms $x^{k}$ in the graph polynomial $da(G;x,y)$ by $ [ x^{k} ] da(G;x,y)$ and the coefficient of the term $x^{k} y^{l}$ by $ [ x^{k} y^{l} ] da(G;x,y)$.
We say that a graph $G$ is characterized by a graph polynomial $f$ if for every graph $G$ such that $f(G) =
f(H)$ we have that $G$ is isomorphic to $H$. The class of graphs $K$ is characterized by a graph polynomial $f$ if
every graph $G \in K$ is characterized by $f$.
Also, when we say a vertex set $S$ contributes a term $t$, we mean the set $S$ induces a connected subgraph $G[S]$ which yields the term $t$ in $da(G;x,y)$.

\section{Definition and relations with other graph polynomials} \label{Section: Definition and relations with other polynomials}
\begin{definition}\label{Definition: defensive alliance polynomial}
	The mappings $f_{x}$ and $f_{y}$ are defined as follows:
	\begin{align*} \label{eq1}		
	& f_{x} :\mathbb{P}(V(G)) \mapsto \mathbb{N} \text{ with }  f_{x}(S) = |S| \text{ and }\\
	& f_{y}:\mathbb{P}(V(G)) \mapsto \mathbb{Z} \text{ with } f_{y}(S) = \min_{u \in S} \{\delta_{S}(u) - \delta_{\bar{S}}(u) + n\} \text{.}	
	\end{align*}
	The \index{defensive alliance polynomial}\emph{defensive alliance polynomial} denoted by \emph{da} is: \label{da definition}
	\[
	da(G;x,y) = \sum_{  \; S \subseteq V(G) }[S \text{ is not empty }]  [ \; G[S] \text{ is connected }]x^{f_{x}(S)} y^{f_{y}(S)}.
	\]
\end{definition}

\subsection{Alliance polynomial}
The \emph{alliance polynomial} defined in \cite{Carb14} denoted by $A$ is:
\[
A(G;y) = \sum_{S \subseteq V(G) }[S \text{ is not empty }][  \; G[S] \text{ is connected }]y^{f_{y}(S)}.
\]

\begin{proposition}
	$A(G;y) \, = \, da(G;1,y) \text{.}$
\end{proposition}

\subsection{Strong alliance polynomial}
\begin{proposition}\label{Definition: strong alliance polynomial}
	Let $S$ be a non-empty subset of $V(G)$ which induces a connected subgraph in $G$. $S$ is \emph{strong defensive alliance} if $f_{y}(S) \geq n$.
\end{proposition}
\begin{definition}
	The \emph{strong alliance polynomial} defined in \cite{Carb16} denoted by  \emph{$a$} is:
	\begin{align*}
	a(G;x) = \sum_{S \subseteq V(G) } &[S \text{ is not empty }][ \; G[S] \text{ is connected }]\\
	&[ S \text{ is strong defensive alliance }]x^{f_{x}(S)}\\
	\end{align*}
	
\end{definition}
\begin{proposition}
	$a(G;x) = \sum_{k=0}^{n-1}[y^{n+k}]da(G;x,y) \text{.}$
\end{proposition}

\subsection{Induced connected subgraph polynomial}
\begin{definition}\label{Definition: Induced connected subgraph polynomial}
	The \index{induced connected subgraph polynomial}\emph{induced connected subgraph polynomial} defined in \cite{Titt11} denoted by \emph{q} is:
	\[
	q(G;x) = \sum_{S \subseteq V(G) }[S \text{ is not empty }][  \; G[S] \text{ is connected }]x^{f_{x}(S)} \text{.}
	\]
\end{definition}
\begin{proposition}
	$q(G;x) = da(G;x,1) \text{.}$
\end{proposition}

\section{Properties} \label{Section: Properties}
\begin{proposition}\label{Proposition: number of connected induced subgraphs}
	The number of connected induced subgraphs of order $k$ is \\ $[x^{k}]da(G;x,1)$.
\end{proposition}
\begin{proof}
	A set $S$ where $S \subseteq V(G)$, contributes a term with $f_{x}(S)=k$ if and only if $S$ induces a connected subgraph in $G$ and $|S|=k$. By substituting $y=1$ in $da(G;x,1)$, we sum the terms with the similar exponent of $x$. Hence, the coefficient of $x^{k}$ in $da(G;x,1)$ is the number of connected induced subgraphs of order $k$ in $G$.
\end{proof}

\begin{proposition}\label{Proposition: Order}
	The order of $G$ is $[x^{1}]da(G;x,1)$.
\end{proposition}
\begin{proof}
	By putting $k=1$ in Proposition \ref{Proposition: number of connected induced subgraphs} we get $[x^{1}]da(G;x,1)$ as the number of connected subgraphs of order one, hence the order of $G$.
\end{proof}

\begin{proposition}\label{Proposition: Size}
	The size of $G$ is $[x^{2}]da(G;x,1)$.
\end{proposition}
\begin{proof}
	By putting $k=2$ in Proposition \ref{Proposition: number of connected induced subgraphs} we get $[x^{2}]da(G;x,1)$ as the number of connected subgraphs of order two, hence the size of $G$.
\end{proof}

\begin{proposition}\label{Proposition: Connected}
	$G$ is connected if and only if $deg_{x} ( da(G;x,y) )=n$.
\end{proposition}
\begin{proof}
	If $G$ is connected, then $V(G)$ contributes the term $x^{n}y^{f_{y}(V(G))}$. Since $G$ has only one subset of vertcies with cardinality $n$ this implies $deg_{x} ( da(G;x,y) )=n$.
	
	Now we prove the converse. If there exists a term in $da(G;x,y)$ where the exponent of $x$ equals $n$, then there exists a connected induced subgraph with order $n$. Since $V(G)$ is the only such subgraph, therefore $G$ is connected.
\end{proof}

\begin{proposition}\label{Proposition: Degree sequence}
	Let $k$ be an integer in the range $0 \leq k \leq n-1$. The number of vertices in $G$ with a degree $k$ is $[xy^{n-k}]da(G;x,y)$. Hence the degree sequence of $G$ can be obtained.
\end{proposition}
\begin{proof}
	Let $v$ be a vertex in $G$. The set $\{v\}$ induces a connected subgraph in $G$ which contributes the term $xy^{n-deg(v)}$ in $da(G;x,y)$. Hence $[xy^{n-deg(v)}]da(G;x,y)$ yields the number of all vertices with degree equal to $deg(v)$.
\end{proof}

\begin{proposition}
	Let $G$ be a simple graph. The maximum order of a component of $G$ is $deg(da(G;x,1))$. Further, the number of components with maximum order $c$ is $[x^{c}]da(G;x,1)$.
\end{proposition}
\begin{proof}
	From the definition of the defensive alliance polynomial, we can see that $deg(da(G;x,1))$ is the order of the maximum component of $G$. Let $c=deg(da(G;x,1))$ and $A=\{ S: |S|=c \text{ and } S \text{ induces a component in } G\}$. Every set $S$ in $A$ contributes a term $x^{c}y^{f_{y}(S)}$ in $da(G;x,y)$. The number of these terms is $|A|$ which can be obtained from $[x^{c}]da(G;x,1)$.
\end{proof}

A vertex in $G$ whose removal results in increase of the number of components of $G$ is a \emph{cut vertex}.
\begin{proposition}\label{Proposition: number of cut vertices}
	Let $G$ be a simple connected graph. The number of cut vertices in $G$ is $n-[x^{n-1}]da(G;x,1)$.
\end{proposition}
\begin{proof}
	Let $v$ be a vertex in $V(G)$, every subset of $V(G)\setminus \{ v \}$ contributes a connected subgraph in $G$ if and only if $v$ is not a cut vertex. Every such set $V(G) \setminus \{ v \}$, contributes a term in $da(G;x,1)$ where the exponent of $x$ is $n-1$. The number of cut vertices is the order minus the sum of the above terms $=n-[x^{n-1}]da(G;x,1)$.
\end{proof}

\begin{proposition}
	Let $G_{1}$, $G_{2}$, $\cdots$, $G_{k}$	be pairwise disjoint graphs. Then
	\begin{equation*}
	da(\cup_{i=1}^{k}G_{i};x,y) = \left( \sum_{i=1}^{k}
	\frac{da(G_{i};x,y)}{y^{|G_{i}|}} \right) y^{\sum_{i=1}^{k}|G_{i}|}.
	\end{equation*}
\end{proposition}
\begin{proof}
	Let $i$ and $j$ be integers in the range $1,2,\cdots, k$. Every connected subgraph in $G_{i}$ is disjoint from subgraphs in $G_{j}$ where $i \neq j$. But the exponent of $y$ in $da(G_{i};x,y)$ is added to $|G_{i}|$, hence the sum of the orders of all the other graphs must be added. 
\end{proof}

\section{Defensive alliance polynomial of special classes of graphs and their characterization by it} \label{Section: Defensive alliance polynomial of special classes of graphs and their characterization by it}
\subsection{The path graph}\label{Subsection: Path}
\begin{proposition}
	A simple graph $G$ is isomorphic to the path $P_{n}$
	if and only if  
	\[
	da(G;x,y) = 2xy^{n-1} + (n-2)xy^{n-2} + y^{n}\sum_{i=2}^{n-1}(n-i+1)x^{i} + x^{n}y^{n+1} \text{, where } n \geq 2.
	\]	
\end{proposition}
\begin{proof}
	First, we show that a graph $G$ which is isomorphic to a path $P_{n}$, has the given defensive alliance polynomial. Let $G$ be of the form in Figure \ref{figure: defensive allaince the path}.
	\begin{figure}[ht!]
		\centering
		
		\begin{tikzpicture}
		\tikzset{vertex/.style = {shape=circle,draw,minimum size=2.5em}}
		
		\node[vertex] (a) at  (0,0) {$v_{1}$};
		\node[vertex] (a1) at (1.5,0) {$v_{2}$};
		\node[vertex] (a2) at (3,0) {$v_{3}$};
		\node[vertex] (c) at  (8,0) {$v_{n-1}$};
		\node[vertex] (c2) at  (10,0) {$v_{n}$};
		
		\draw (a)  to (a1);
		\draw (a1) to (a2);
		
		\path (a2) to node {\dots} (c);
		\node [shape=circle,minimum size=1.5em] (a3) at (4.5,0) {};
		\draw (a2) to (a3);
		
		\node [shape=circle,minimum size=1.5em] (c1) at (6.5,0) {};
		\draw (c) to (c1);
		\draw (c) to (c2);
		\end{tikzpicture}
		\caption{A path graph}
		\label{figure: defensive allaince the path}
	\end{figure}
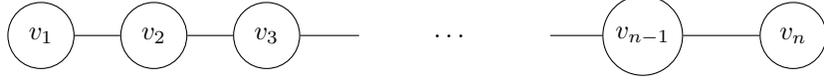
	
	The non-empty subsets of $V(G)$ which induce connected subgraphs in $G$, can be partitioned into the following parts: 
	The part $\{ \{ v_{1} \} , \{ v_{n} \} \}$ in which each set contributes the term $xy^{n-1}$ and by summing, we get the term $2xy^{n-1}$. The part $\{ \{ v_{2} \} , \{ v_{3} \}, \cdots , \{ v_{n-1} \} \}$ in which each set contributes the term $xy^{n-2}$ and by summing, we get the term $(n-2)xy^{n-2}$. The part containing the sets of cardinality $i$ in the range of $2 \leq i \leq n-1$ in which each set contributes the term $x^{i}y^{n}$. By adding the terms we get 
	\begin{align*}
	&(n-1)x^{2}y^{n} + (n-2)x^{3}y^{n} + \cdots + (n-(n-2))x^{n-1}y^{n}\\
	& = y^{n}\sum_{i=2}^{n-1}(n-i+1)x^{i} \text{.}
	\end{align*}
	Finally, the part containing $V(G)$ in which $V(G)$ contributes the term $x^{n}y^{n+1}$.	
	
	Now we prove the converse. Let $n$ be an integer where $n\geq 2$, and $H$ is a graph with the defensive alliance polynomial,
	\begin{equation*}
	da(H;x,y) = 2xy^{n-1} + (n-2)xy^{n-2} + y^{n}\sum_{i=2}^{n-1}(n-i+1)x^{i} + x^{n}y^{n+1} \text{.}
	\end{equation*}	
	By Proposition \ref{Proposition: Order}, the order of $H$ equals $n$. By Proposition \ref{Proposition: Size}, the size of $H$ equals $n-1$. By Proposition \ref{Proposition: Connected}, $H$ is connected. Hence, $H$ is a tree. By Proposition \ref{Proposition: Degree sequence}, the degree sequence of $H$ is $(2,2, \cdots ,2,1,1)$.
	Consequently, $H$ is isomorphic to the path graph $P_{n}$.
\end{proof}

\subsection{The cycle graph}\label{Subsection: Cyclea}
\begin{proposition}
	A simple graph $G$ is isomorphic to the cycle $C_{n}$
	if and only if  
	\[
	da(G;x,y) = nxy^{n-2} + ny^{n}\sum_{i=2}^{n-1}x^{i} + x^{n}y^{n+2} 
	\text{, where } n \geq 3.
	\]
\end{proposition}
\begin{proof}
	First, we show that a graph $G$ which is isomorphic to a cycle $C_{n}$, has the given defensive alliance polynomial.
	
	The non-empty subsets of $V(G)$ which induce connected subgraphs in $G$, can be partitioned into the following parts: 
	The part containing the sets of cardinality one in which each set contributes the term $xy^{n-2}$ and by summing, we get the term $nxy^{n-2}$. The part containing the sets of cardinality $i$ in the range of $2 \leq i \leq n-1$ in which each set contributes the term $x^{i}y^{n}$. By adding the terms we get 
	\begin{align*}
	&nx^{2}y^{n} + nx^{3}y^{n} + \cdots + nx^{n-1}y^{n} \\
	& = ny^{n}\sum_{i=2}^{n-1}x^{i} \text{.}
	\end{align*}
	Finally, the part containing $V(G)$ in which $V(G)$ contributes the term $x^{n}y^{n+2}$.	
	
	Now we prove the converse. Let $n$ be an integer where $n\geq 3$, and $H$ is a graph with the defensive alliance polynomial 
	\[
	da(H;x,y) = nxy^{n-2} + ny^{n}\sum_{i=2}^{n-1}x^{i} + x^{n}y^{n+2} \text{.}
	\]
	By Proposition \ref{Proposition: Order}, the order of $H$ equals $n$. By Proposition \ref{Proposition: Connected}, $H$ is connected. By Proposition \ref{Proposition: Degree sequence}, the degree sequence of $H$ is $(2,2, \cdots ,2)$.
	
	Consequently, $H$ is isomorphic to the cycle graph $C_{n}$.
\end{proof}

\subsection{The star graph}\label{Subsection: Star}
\begin{definition}
	Let $n$ be a positive integer. The \index{star graph}\emph{star graph} denoted by $S_{n}$ is defined by the graph join $nK_{1}+K_{1}$. Further the vertex with the maximum degree is called the \emph{center}.
\end{definition}
\begin{proposition}
	A simple graph $G$ is isomorphic to the star $S_{n}$
	if and only if  
	\[
	da(G;x,y) = xy + nxy^{n-1} + \sum_{i=1}^{ \lfloor \frac{n}{2} \rfloor } \binom{n}{i}x^{i+1}y^{2i} + \sum_{i= \lceil \frac{n+1}{2} \rceil }^{n} \binom{n}{i}x^{i+1}y^{n+1}
	\text{, where }n \geq 1 \text{.}
	\]
\end{proposition}
\begin{proof}
	First, we show that a graph $G$ which is isomorphic to a star $S_{n}$, has the given defensive alliance polynomial. Let $G$ be of the form in Figure \ref{figure: defensive allaince the star}.
	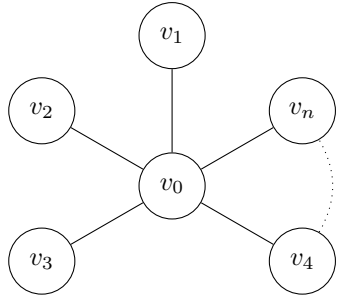
\begin{figure}[ht!]
		\centering
		\begin{tikzpicture}[baseline=(current bounding box.north)]
		\tikzset{vertex/.style = {shape=circle,draw,minimum size=2.5em}}
		
		\node[vertex] (v0) at  (0,0) {$v_{0}$};
		\node[vertex] (v1) at (90:2) {$v_{1}$};
		\node[vertex] (v2) at (150:2) {$v_{2}$};
		\node[vertex] (v3) at  (210:2) {$v_{3}$};
		\node[vertex] (v4) at  (330:2) {$v_{4}$};
		\node[vertex] (vn) at  (30:2) {$v_{n}$};

		\draw [dotted,bend right] (v4) to (vn);
		
		\draw (v0) to (v1);
		\draw (v0) to (v2);
		\draw (v0) to (v3);
		\draw (v0) to (v4);
		\draw (v0) to (vn);
		
		\end{tikzpicture}
		\caption{A star graph}
		\label{figure: defensive allaince the star}
	\end{figure}
	
	The non-empty subsets of $V(G)$ which induce connected subgraphs in $G$, can be partitioned into the following parts: 
	The part $\{ \{ v_{0} \} \}$ in which $\{ v_{0} \}$ contributes the term $xy$. The part $\{ \{ v_{1} \},\{ v_{2} \},\cdots , \{ v_{n} \} \}$  in which each set contributes the term $xy^{n-1}$ and by summing, we get the term $nxy^{n-1}$. The part containing the sets of cardinality $i$ in the range of $2 \leq i \leq \lfloor \frac{n}{2} \rfloor $ in which each set contributes the term $x^{i+1}y^{2i+1}$ and by summing, we get $\sum_{i=1}^{ \lfloor \frac{n}{2} \rfloor } \binom{n}{i} x^{i+1}y^{2i+1}$. The part containing the sets of cardinality $i$ in the range of $ \lceil  \frac{n+1}{2} \rceil \leq i \leq n$ in which each set contributes the term $x^{i+1}y^{n+1}$ and by summing, we get $\sum_{ \lceil \frac{n+1}{2} \rceil }^{i=n} \binom{n}{i} x^{i+1}y^{n+1}$.
	
	Now we prove the converse. Let $n$ be an integer where $n\geq 1$, and $H$ is a graph with the defensive alliance polynomial 
	\[
	da(H;x,y) = xy + nxy^{n-1} + \sum_{i=1}^{ \lfloor \frac{n}{2} \rfloor } \binom{n}{i}x^{i+1}y^{2i} + \sum_{i= \lceil \frac{n+1}{2} \rceil }^{n} \binom{n}{i}x^{i+1}y^{n+1} \text{.}
	\]
	By Proposition \ref{Proposition: Order}, the order of $H$ equals $n+1$. By Proposition \ref{Proposition: Size}, the size of $H$ equals $n$. By Proposition \ref{Proposition: Connected}, $H$ is connected. Hence, $H$ is a tree. By Proposition \ref{Proposition: Degree sequence}, the degree sequence of $H$ is $(n,1,1 \cdots ,1)$.
	Consequently, $H$ is isomorphic to the star graph $S_{n}$.
\end{proof}

\subsection{The complete graph}\label{Subsection: Complete}
\begin{proposition}
	A simple graph $G$ is isomorphic to the complete graph $K_{n}$
	if and only if  
	\[
	da(G;x,y) = \frac{(1+xy^{2})^{n}-1}{y} 
	\text{, where }n \geq 1 \text{ and } y \neq 0 \text{.}
	\]
\end{proposition}
\begin{proof}
	First, we show that a graph $G$ which is isomorphic to a complete graph $K_{n}$, has the given defensive alliance polynomial.
	
	The non-empty subsets of $V(G)$ which induce connected subgraphs in $G$, can be partitioned into one part: 
	The part containing the sets of cardinality $i$ in the range of $1 \leq i \leq n$ in which each set contributes the term $x^{i}y^{2i -1}$ and by summing, we get:
	\begin{align*}
	&\binom{n}{1}x^{1}y^{1} +\binom{n}{2}x^{2}y^{3} + \cdots + \binom{n}{n}x^{n}y^{2n-1} \\
	&=\sum_{i=1}^{n} \binom{n}{i}x^{i}y^{2i-1}  \\
	&=\frac{1}{y} ( \sum_{i=0}^{n} \binom{n}{i}(xy^{2})^{i} - 1 ) \\
	&= \frac{(1+xy^{2})^{n}-1}{y} \text{.} 
	\end{align*}
	
	Now we prove the converse. Let $n$ be an integer where $n\geq 1$ and $H$ is a graph with the defensive alliance polynomial,
	$da(H;x,y) = \frac{(1+xy^{2})^{n}-1}{y}$.
	By Proposition \ref{Proposition: Order}, the order of $H$ equals $n$. By Proposition \ref{Proposition: Degree sequence}, the degree sequence of $H$ is $(n-1,n-1, \cdots ,n-1)$.
	Consequently, $H$ is isomorphic to the complete graph $K_{n}$.
\end{proof}

\subsection{The regular graph}\label{Subsection: Regular characterization proofs and formula}
\begin{proposition}
	A simple graph $G$ is isomorphic to a $\Delta$-regular graph
	if and only if $[x]da(G;x,y)=ny^{n-\Delta}$.
\end{proposition}
\begin{proof}
	First, we show that a graph $G$ which is isomorphic to a $\Delta$-regular graph has $[x]da(G;x,y)=ny^{n-\Delta}$.
	Every subset of $V(G)$ which induces a connected subgraph in $G$, contributes a term $xy^{n-\Delta}$ and by summing, we get the term $nxy^{n-\Delta}$.
	
	Now we prove the converse. Let $H$ be a graph with $[x]da(G;x,y)=ny^{n-\Delta}$.
	By Proposition \ref{Proposition: Order}, the order of $H$ equals $n$. By Proposition \ref{Proposition: Degree sequence}, the degree sequence of $H$ is $(\Delta,\Delta, \cdots ,\Delta)$.
	Consequently, $H$ is isomorphic to a $\Delta$-regular graph.
\end{proof}

\begin{lemma}\label{lemma:regularGraph_setComponenets}
	Let $G$ be a $\Delta$-regular graph. A subset of $V(G)$ of cardinality $k$ induces a component in $G$ if and only if it contributes in $da(G;x,y)$ a term $x^{k}y^{\Delta+n}$.
\end{lemma}
\begin{proof}
	Every component of order $k$ in a $\Delta$-regular graph, contributes a term with $x^{k}y^{\Delta+n}$.
	
	To prove the converse, let $S$ be a subset of $V(G)$ of cardinality $k$ which contributes in $da(G;x,y)$ a term $x^{k}y^{\Delta+n}$. For sake of contradiction, assume that $S$ is not a component. Hence, there is a vertex in $S$ which is connected to other vertices outside $S$. Let the maximum number of vertices connected to a vertex in $S$ from outside of $S$ to be $t$. Hence $S$ contributes in $da(G;x,y)$ a term $x^{k}y^{n + (\Delta-t)-t}=x^{k}y^{n + \Delta-2t}$, contradiction since $t \neq 0$. Consequently, $t = 0$ and $S$ contributes a component in $G$.  
\end{proof}

\begin{lemma}\label{lemma:regularGraph_componenets}
	For a $\Delta$-regular graph $G$, the number of components with cardinality $k$ is $=[x^{k}y^{\Delta+n}]da(G;x,y)$.
\end{lemma}
\begin{proof}
	From Lemma \ref{lemma:regularGraph_setComponenets}, every subset of $V(G)$ with cardinality $k$, induces a component in $G$ if and only if this subset contributes in $da(G;x,y)$ a term $x^{k}y^{\Delta+n}$. By summing the terms, the result follows. 
\end{proof}
\begin{corollary}
	Let $G$ be a connected $\Delta$-regular graph. $[x^{n}]da(G;x,y)=y^{\Delta+n}$.
\end{corollary}
\begin{proof}
	From Lemma \ref{lemma:regularGraph_componenets}, the result follows.
\end{proof}

\subsection{The double star graph}\label{Subsection: Double star}
\begin{definition}
	Let $r$ and $t$ be positive integers. The \index{double star graph}\emph{star graph} denoted by $S_{r,t}$ is defined by the graph union $S_{r} \cup S_{t}$ and connecting the two centers of the two stars.
\end{definition}
\begin{proposition}
	A simple graph $G$ is isomorphic to the double star $S_{r,t}$
	if and only if 
	\begin{align*}
		[x]da(G;x,y) =& (r+t)y^{r+t+1}+y^{r+1}+y^{t+1} \text{ and}\\
		[x^{r+t+2}]da(G;x,y) =& y^{r+t+3} \text{, where $r$ and $t$ are positive integers.}
	\end{align*}
\end{proposition}
\begin{proof}
	First, we show that a graph $G$ which is isomorphic to a double star $S_{r,t}$, has the above properties in the proposition. Let $G$ be of the form in Figure \ref{figure: defensive allaince the double star}.
	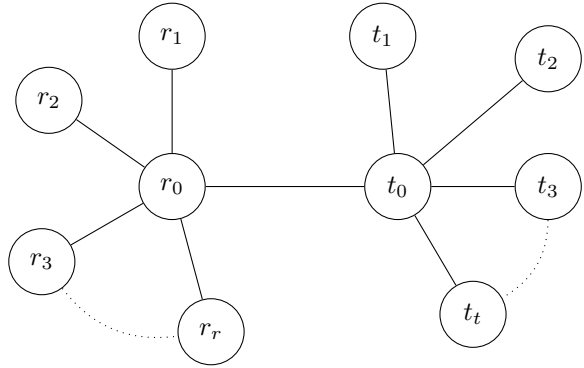
\begin{figure}[ht!]
		\centering
		\begin{tikzpicture}[baseline=(current bounding box.north)]
		\tikzset{vertex/.style = {shape=circle,draw,minimum size=2.5em}}
		
		\node[vertex] (r0) at  (0,0) {$r_{0}$};
		\node[vertex] (r1) at (0,2) {$r_{1}$};
		\node[vertex] (r2) at (145:2) {$r_{2}$};
		\node[vertex] (r3) at  (210:2) {$r_{3}$};
		\node[vertex] (r4) at  (285:2) {$r_{r}$};
		
		\node[vertex] (t0) at  (3,0) {$t_{0}$};
		\node[vertex] (t1) at (2.8,2) {$t_{1}$};
		\node[vertex] (t2) at (5,1.7) {$t_{2}$};
		\node[vertex] (t3) at  (5,0) {$t_{3}$};
		\node[vertex] (t4) at  (4,-1.7) {$t_{t}$};

		\draw (r0)  to (r1)
		(r0) to (r2)
		(r0) to (r3)
		(r0) to (r4)
		;
		\draw (t0)  to (t1)
		(t0) to (t2)
		(t0) to (t3)
		(t0) to (t4)
		;	
		\draw (t0)  to (r0);
		
		\draw [dotted,bend right] (r3) to (r4);
		\draw [dotted,bend left] (t3) to (t4);
		\end{tikzpicture}
		\caption{A double star graph}
		\label{figure: defensive allaince the double star}
	\end{figure}
	
	The subsets of $V(G)$ with cardinality one which induce connected subgraphs in $G$, can be partitioned into the following parts: The part $\{ \{r_{0} \} \}$ in which $\{ r_{0} \}$ contributes the term $xy^{t+1}$. The part $\{ \{t_{0} \} \}$ in which $\{ t_{0} \}$ contributes the term $xy^{r+1}$. The part $\{ \{r_{1}\},\{r_{2}\}, \cdots ,\{r_{r}\},  \{t_{1}\},\{t_{2}\}, \cdots ,\{t_{t}\}\}$ in which each set contributes the term $xy^{r+t+1}$ and by summing, we get the term $(r+t)xy^{r+t+1}$.
	
	The set $V(G)$ contributes the term $x^{r+t+2}y^{r+t+3}$.
	
	Now we prove the converse. Let $r$ and $t$ be integers and $H$ is a graph with 	
	\begin{align*}
		[x]da(G;x,y) =& (r+t)y^{r+t+1}+y^{r+1}+y^{t+1} \text{ and}\\
		[x^{r+t+2}]da(G;x,y) =& y^{r+t+3} \text{.}
	\end{align*}	
	
	By Proposition \ref{Proposition: Order}, the order of $H$ equals $r+t+2$. By Proposition \ref{Proposition: Connected}, $H$ is connected. By Proposition \ref{Proposition: Degree sequence}, the degree sequence of $H$ is $(r+1,s+1,1,1, \cdots ,1)$. Let the vertex with degree $r+1$ be $r_{0}$ and the vertex with degree $t+1$ be $t_{0}$. Connect $r_{0}$ with $r+1$ vertices. If all those vertices connected to $r_{0}$ are with degree one, then the graph will be disconnected which is contradiction. Then $r_{0}$ is connected to $t_{0}$. By connecting the rest of the vertcies to $t_{0}$, $H$ is reconstructed. Consequently, $H$ is isomorphic to the double star graph $S_{r,t}$.
\end{proof}

\bigskip

\subsection{The complete bipartite graph}\label{Subsection: Complete bipartite}
\begin{lemma}\label{lemma: bipartiteRegular1}
	Let $G$ be a simple graph. Let $k_{3}$ be the number of the subsets which induce connected subgraphs in $G$ with order three. Let the number of connected subgraphs in $G$ with order three and size two be $S_{3,2}$ and with order three and size three be $S_{3,3}$, then
	\[
	k_{3} = S_{3,2} - 2S_{3,3}\text{.}
	\]
\end{lemma}
\begin{proof}
	Any induced connected subgraph in $G$ with order three will be isomorphic either to a cycle or a path of order three. If the induced connected subgraph in $G$ with order three is a cycle then it will count three subgraphs which are isomorphic to a path of order three.
\end{proof}

\begin{lemma}\label{lemma:bipartiteRegular2}
	Let $G$ be a $\Delta$-regular simple graph. then
	\[
	S_{3,2} = n\binom{\Delta}{2}\text{.}
	\]
\end{lemma}
\begin{proof}
	The number of connected subgraphs in $G$ with order three and size two containing a specific vertex $v$ as the common vertex between the two edges is formed by choosing any two vertices from the neighbors is $=\binom{\Delta}{2}$. By multiplying with the number of all vertices $n$, the result follows.
\end{proof}
\begin{lemma}\label{lemma:bipartiteRegular3}
	Let $G$ be a $\Delta$-regular connected simple graph with order $2 \Delta$. $G$ is isomorphic to  $K_{\Delta,\Delta}$ if and only if $k_{3} = n \binom{\Delta}{2}$.
\end{lemma}
\begin{proof}
	First, we show that if a graph $G$ is isomorphic to $K_{\Delta,\Delta}$ then $k_{3} = n \binom{\Delta}{2}$.
	
	$G$ is isomorphic to $K_{\Delta,\Delta}$ then $G$ has no cycles of order three. By Lemma \ref{lemma: bipartiteRegular1} and Lemma \ref{lemma:bipartiteRegular2}, the result follows.\\
	
	Now we prove the converse. 
	By Lemma \ref{lemma: bipartiteRegular1}, $k_{3} = S_{3,2}$ and $S_{3,3}=0$. $G$ is free of cycles of order three. Any vertex $v$ is adjacent to $\Delta$ pairwise nonadjacent vertices which have a degree $\Delta$ and need to be adjacent to $\Delta-1$ other vertices which are not adjacent to $v$. By constructing the graph, we obtain that $G$ is isomorphic to $K_{\Delta,\Delta}$.
\end{proof}

\begin{proposition}
	A simple graph $G$ is isomorphic to the complete bipartite graph $K_{n,m}$
	if and only if  
	\[
	da(G;x,y) = nxy^{n} + mxy^{m} + y^{n+m}\sum_{i=1}^{n}\sum_{j=1}^{m} \binom{n}{i}\binom{m}{j} x^{i+j}y^{\min\{2i-n,2j-m\}} \text{,}
	\]
	where $n,m$ are positive integers.
\end{proposition}
\begin{proof}
	First, we show that a simple graph $G$ which is isomorphic to the complete bipartite graph $K_{n,m}$, has the given defensive alliance polynomial. Let $K_{n,m}$ be of the form $G(U \cup W,E)$ where $|U|=n$, $|W|=m$ and $U,W$ are the parts of $K_{n,m}$. 
	
	The non-empty subsets of $V(G)$ which induce connected subgraphs in $G$, can be partitioned into the following parts: 
	The part containing the sets of cardinality one from $U$ in which each set contributes the term $xy^{n}$ and by summing, we get the term $nxy^{n}$. The part containing the sets of cardinality one from $W$ in which each set contributes the term $mxy^{m}$ and by summing, we get the term $mxy^{m}$. The part containing the sets of cardinality more than one in which we choose subset of cardinality $i$ from $U$ and another subset of cardinality $j$ from $W$ which contributes the term $y^{n+m}\left(x^{i+j}y^{\min\{2i-m,2j-n\}}\right)$ and by summing, we get the term \\ $y^{n+m}\sum_{i=1}^{n}\sum_{j=1}^{m} \binom{n}{i}\binom{m}{j} x^{i+j}y^{\min\{2j-m,2i-n\}}$.
	
	Now we prove the converse. Let $n,m$ be positive integers, and $H$ is a graph with
	\[
	da(H;x,y) = nxy^{n} + mxy^{m} + y^{n+m}\sum_{i=1}^{n}\sum_{j=1}^{m} \binom{n}{i}\binom{m}{j} x^{i+j}y^{\min\{2i-n,2j-m\}}.
	\]
	By Proposition \ref{Proposition: Order}, the order of $H$ equals $n+m$. By Proposition \ref{Proposition: Size}, the size of $H$ equals $nm$. By Proposition \ref{Proposition: Connected}, $H$ is connected. By Proposition \ref{Proposition: Degree sequence}, the degree sequence of $H$ is $(n,n,\cdots,n,m,m, \cdots ,m)$.
	Partition $V(H)$ into two sets $W,U$ where $W$ contains all vertices with degree $n$ and $U$ contain all vertices of degree $m$.
	\begin{itemize}
		\item Case $1$: $n \neq m$, assume $n > m$.	Note that $[x^{2}y^{m+2}]da(G;x,y)=0$, since this happens only if there is no edge between two vertices with degree $n$. By counting the edges and joining the vertices from $W$ to $U$, $H$ is isomorphic to $K_{n,m}$
		\item Case $2$: $n=m$ then $H$ is regular. Note that:
		\begin{align*}
			k_{3} =& [x^{3}]da(G;x,1)  \\
			=& 2 n \binom{n}{2} \text{.}
		\end{align*}
		Consequently, by Lemma \ref{lemma:bipartiteRegular3}, $H$ is isomorphic to the complete bipartite graph $K_{n,n}$.
	\end{itemize}
\end{proof}

\subsection{The wheel graph}\label{Subsection: Wheel}
\begin{definition}
	Let $n$ be a positive integer larger than three. The \index{wheel graph}\emph{wheel graph} denoted by $W_{n}$ is defined by the graph join $C_{n}+K_{1}$.
\end{definition}
\begin{proposition}
	A simple graph $G$ is isomorphic to the wheel $W_{n}$ if and only if
	\begin{align*}
		[x]da(G;x,y) =& ny^{n-2}+ y \text{ and} \\ 
		[x^{n}]da(G;x,y) =& (n+1)y^{n+2}  \text{ and} \\
		[x^{n+1}]da(G;x,y) =& y^{n+4} \text{, where $n \geq 3$.}
	\end{align*}
	
\end{proposition}
\begin{proof}
	First, we show that a graph $G$ which is isomorphic to a wheel $W_{n}$ has the above properties in the proposition. Let $G$ be of the form in Figure \ref{figure: defensive allaince the wheel}.
	\begin{figure}[ht!]
		\centering
		\begin{tikzpicture}[baseline=(current bounding box.north)]
		\tikzset{vertex/.style = {shape=circle,draw,minimum size=2.5em}}
		
		\draw (330:2) arc (-30:210:2);
		\draw [dotted] (210:2) arc (210:330:2);
		\node[vertex,fill=white] (v0) at  (0,0) {$v_{0}$};
		\node[vertex,fill=white] (v1) at (90:2) {$v_{1}$};
		\node[vertex,fill=white] (v2) at (150:2) {$v_{2}$};
		\node[vertex,fill=white] (v3) at  (210:2) {$v_{3}$};
		\node[vertex,fill=white] (v4) at  (330:2) {$v_{n-1}$};
		\node[vertex,fill=white] (vn) at  (30:2) {$v_{n}$};


		\draw (v0) to (v1);
		\draw (v0) to (v2);
		\draw (v0) to (v3);
		\draw (v0) to (v4);
		\draw (v0) to (vn);
		
		\end{tikzpicture}
		\caption{A wheel graph}
		\label{figure: defensive allaince the wheel}
	\end{figure}
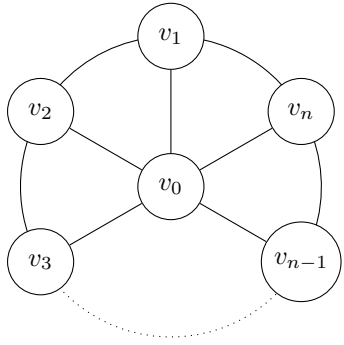
	
	The subsets of $V(G)$ with cardinality one which induce connected subgraphs in $G$, can be partitioned into the following parts: The part $\{ \{ v_{0} \} \}$ in which $\{ v_{0} \}$ contributes the term $xy$. The part $\{ \{ v_{1} \},\{ v_{2} \}, \cdots , \{ v_{n} \} \}$ in which every set contributes the term $xy^{n-2}$ and by summing, we get the term $nxy^{n-2}$.
	
	The set $V(G)$ contributes the term $x^{n+1}y^{n+4}$. And if we delete any vertex from $V(G)$, we get a set which contributes the term $x^{n}y^{n+2}$ and by summing, we get $(n+1)x^{n}y^{n+2}$.
	
	Now we prove the converse. Let $n$ be an integer, $n\geq 3$, and $H$ is a graph with	
	\begin{align*}
		[x]da(H;x,y) =& ny^{n-2}+ y \text{ and} \\ 
		[x^{n}]da(H;x,y) =& (n+1)y^{n+2}  \text{ and} \\
		[x^{n+1}]da(H;x,y) =& y^{n+4} \text{.}
	\end{align*}
	
	By Proposition \ref{Proposition: Order}, the order of $H$ equals $n+1$. By Proposition \ref{Proposition: Connected}, $H$ is connected. By Proposition \ref{Proposition: Degree sequence}, the degree sequence of $H$ is $(n,3,3, \cdots ,3)$. By Proposition \ref{Proposition: number of cut vertices}, the number of cut vertices is zero. Hence all the subgraphs $G \setminus \{v\}$ where $v \in V(G)$, are all connected. Let $v_{0}$ be the vertex with degree $n$. The specific graph $G \setminus \{v_{0}\}$ is connected and with degree sequence $(2,2, \cdots ,2)$ which is isomorphic to the cycle graph $C_{n}$. By connecting the vertex $v_{0}$ to every vertex in $C_{n}$, $H$ is constructed which is isomorphic to the wheel graph $ W_{n}$.	
\end{proof}

\bigskip

\subsection{The open wheel graph}\label{Subsection: open wheel}
\begin{definition}
	Let $n$ be a positive integer larger than two. The \index{open wheel graph}\emph{open wheel graph} denoted by $W_{n}^{'}$ is defined by the graph join $P_{n}+K_{1}$.  This graph is sometimes also known as \emph{Fan}.
\end{definition}
\begin{proposition}
	A simple graph $G$ is isomorphic to the open wheel $W_{n}^{'}$
	if and only if 
	\begin{align*}
		[x]da(G;x,y) =& 2y^{n-1} + (n-2)y^{n-2}+ xy \text{ and} \\ 
		[x^{n}]da(G;x,y) =& 3y^{n+1}+(n-2)y^{n+2} \text{ and} \\
		[x^{n+1}]da(G;x,y) =& y^{n+3} \text{, where $n \geq 4$.}
	\end{align*}
	
\end{proposition}
\begin{proof}
	First, we show that a graph $G$ which is isomorphic to an open wheel $W_{n}^{'}$, has the above properties in the proposition. Let $G$ be of the form in Figure \ref{figure: defensive allaince the Open wheel graph}.
	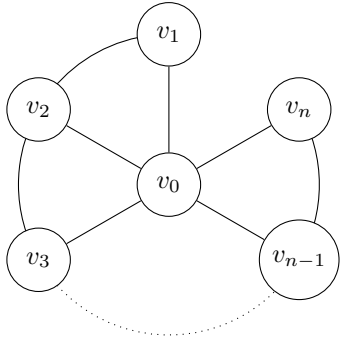
\begin{figure}[ht!]
		\centering	
		\begin{tikzpicture}[baseline=(current bounding box.north)]
		\tikzset{vertex/.style = {shape=circle,draw,minimum size=2.4em}}
		
		\draw (90:2) arc (90:210:2);
		\draw (330:2) arc (-30:30:2);
		\draw [dotted] (210:2) arc (210:330:2);
		
		\node[vertex,fill=white] (v0) at  (0,0) {$v_{0}$};
		\node[vertex,fill=white] (v1) at (90:2) {$v_{1}$};
		\node[vertex,fill=white] (v2) at (150:2) {$v_{2}$};
		\node[vertex,fill=white] (v3) at  (210:2) {$v_{3}$};
		\node[vertex,fill=white] (v4) at  (330:2) {$v_{n-1}$};
		\node[vertex,fill=white] (vn) at  (30:2) {$v_{n}$};		
		
		\draw (v0) to (v1);
		\draw (v0) to (v2);
		\draw (v0) to (v3);
		\draw (v0) to (v4);
		\draw (v0) to (vn);
		\end{tikzpicture}
		\caption{An open wheel graph}
		\label{figure: defensive allaince the Open wheel graph}
	\end{figure}
	
	The subsets of $V(G)$ with cardinality one which induce connected subgraphs in $G$, can be partitioned into the following parts: The part $\{ \{ v_{0} \} \}$ in which $\{ v_{0} \}$ contributes the term $xy$. The part $\{ \{ v_{2} \},\{ v_{3} \}, \cdots , \{ v_{n-1} \} \}$ in which every set contributes the term $xy^{n-2}$ and by summing, we get the term $(n-2)xy^{n-2}$. The part $\{ \{ v_{1} \} \},\{ v_{n} \} \}$ in which each set contributes the term $xy^{n-1}$ and by summing, we get $2xy^{n-1}$.
	
	The set $V(G)$ contributes the term $x^{n+1}y^{n+3}$.
	
	Each of the subsets $ V(G) \setminus \{v_{2}\} , V(G) \setminus \{v_{n-1}\}$ and $V(G) \setminus \{v_{0}\} $ contributes the term $x^{n}y^{n+1}$ and by summing, we get the term $3x^{n}y^{n+1}$. Each subset of cardinality $n$ but not the previous, contributes the term $x^{n}y^{n+2}$ and by summing, we get $(n-2)x^{n}y^{n+2}$.
	
	Now we prove the converse. Let $n$ be an integer, $n\geq 4$, and $H$ is a graph with 
	\begin{align*}
		[x]da(H;x,y) =& 2y^{n-1} + (n-2)y^{n-2}+ xy \text{ and} \\ 
		[x^{n}]da(H;x,y) =& 3y^{n+1}+(n-2)y^{n+2} \text{ and} \\
		[x^{n+1}]da(H;x,y) =& y^{n+3} \text{.} 
	\end{align*}
	By Proposition \ref{Proposition: Order}, the order of $H$ equals $n+1$. By Proposition \ref{Proposition: Connected}, $H$ is connected. By Proposition \ref{Proposition: Degree sequence}, the degree sequence of $H$ is $(n,3,3, \cdots ,3,2,2)$. By Proposition \ref{Proposition: number of cut vertices}, the number of cut vertices is zero. Hence all the graphs $G \setminus \{v\}$ where $v \in V(G)$, are all connected. Let the vertex with degree $n$ be $v_{0}$. The specific subgraph $G \setminus \{v_{0}\}$ is connected and with degree sequence $(2,2, \cdots ,2,1,1)$ which is isomorphic to the path graph $P_{n}$. By connecting the vertex $v_{0}$ to every vertex in $P_{n}$, $H$ is constructed which is isomorphic to the open wheel graph $W_{n}^{'}$.	
\end{proof}

\bigskip

\subsection{The friendship graph}\label{Subsection: friendship graph}
\begin{definition}
	Let $n$ be a positive integer. The \index{friendship graph}\emph{friendship graph} denoted by $F_{n}$ is defined by the graph join $nK_{2}+K_{1}$. This graph is also known as \emph{Windmill graph}.
\end{definition}
\begin{proposition}
	A simple graph $G$ is isomorphic to the friendship $F_{n}$ 
	if and only if  
	\begin{align*}
		[x]da(G;x,y) =& 2ny^{2n-1} +y \text{, where $n$ is a positive integer.}
	\end{align*}
	
\end{proposition}
\begin{proof}
	First, we show that a graph $G$ which is isomorphic to a friendship graph $F_{n}$, has the above properties in the proposition. Let $G$ be of the form in Figure \ref{figure: defensive allaince the friendship}.
	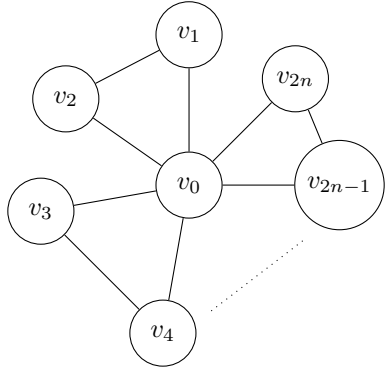
\begin{figure}[ht!]
		\centering
		\begin{tikzpicture}[baseline=(current bounding box.north)]
		\tikzset{vertex/.style = {shape=circle,draw,minimum size=2.5em}}
		
		\node[vertex] (v0) at  (0,0) {$v_{0}$};
		\node[vertex] (v1) at (0,2) {$v_{1}$};
		\node[vertex] (v2) at (145:2) {$v_{2}$};
		\node[vertex] (v3) at  (190:2) {$v_{3}$};
		\node[vertex] (v4) at  (260:2) {$v_{4}$};
		\node[vertex] (vn1) at  (0:2) {$v_{2n-1}$};
		\node[vertex] (vn) at  (45:2) {$v_{2n}$};

		\draw (v0)  to (v1);
		\draw (v0) to (v2);
		\draw (v0) to (v3);
		\draw (v0) to (v4);
		\draw (v0) to (vn);
		\draw (v0) to (vn1);
		
		\draw (v1) to (v2)  (v3) to (v4) (vn1) to (vn);
		
		\draw[dotted] (280:1.7) -- (-25:1.7);
		\end{tikzpicture}
		\caption{A friendship graph}
		\label{figure: defensive allaince the friendship}
	\end{figure}
	
	The subsets of $V(G)$ with cardinality one which induce connected subgraphs in $G$, can be partitioned into the following parts: The part $\{ \{ v_{0} \} \}$ in which $\{ v_{0} \}$ contributes the term $xy$. The part $\{ \{ v_{1} \},\{ v_{2} \}, \cdots , \{ v_{2n} \} \}$ in which every set contributes the term $xy^{2n-1}$ and by summing we get the term $2nxy^{2n-1}$.
	
	Now we prove the converse. Let $n$ be a positive integer, and $H$ is a graph with
	\begin{align*}
		[x]da(H;x,y) =& 2ny^{2n-1} +y \text{.}
	\end{align*}
	By Proposition \ref{Proposition: Order}, the order of $H$ equals $2n+1$. By Proposition \ref{Proposition: Degree sequence}, the degree sequence of $H$ is $(2n,2,2, \cdots ,2)$. We construct the graph by first connecting by an edge the vertex with degree $2n$ to every other vertex. Second every other vertex choose any arbitrary vertex not the one with degree $2n$ and connect it with an edge to complete its degree. Hence the constructed graph $H$ is isomorphic to the friendship graph $ F_{n}$.	
\end{proof}

\bigskip

\subsection{The triangular book graph}\label{Subsection: triangular graph}
\begin{definition}
	Let $n$ be a positive integer. The \index{triangular book graph}\emph{triangular book graph} denoted by $B_{n}$ is defined by the graph join $nK_{1}+K_{2}$.
\end{definition}
\begin{proposition}
	A simple graph $G$ is isomorphic to the triangular book graph $B_{n}$
	if and only if 
	\begin{align*}
		[x]da(G;x,y) =& 2y + ny^{n} \text{, where $n$ is a positive integer.}
	\end{align*}
	
\end{proposition}
\begin{proof}
	First, we show that a simple graph $G$ which is isomorphic to a triangular book graph $B_{n}$, has the above properties in the proposition. Let $G$ be of the form in Figure \ref{figure: defensive allaince the traingluar book}.
	\begin{figure}[ht!]
		\centering
		\begin{tikzpicture}[baseline=(current bounding box.north)]
		\tikzset{vertex/.style = {shape=circle,draw,minimum size=2.5em}}
		
		\node[vertex] (va) at (-2,-1) {$v_{a}$};
		\node[vertex] (vb) at (2,-1) {$v_{b}$};
		\node[vertex] (v1) at  (-1,1) {$v_{2}$};
		\node[vertex] (v2) at  (-3,1) {$v_{1}$};
		\node[vertex] (v3) at  (1,1) {$v_{3}$};
		\node[vertex] (vn) at  (3,1) {$v_{n}$};
		
		\draw (va)  to (v1);
		\draw (va) to (v2);
		\draw (va) to (v3);
		\draw (va) to (vn);
		\draw (vb) to (v1);
		\draw (vb) to (v2);
		\draw (vb) to (v3) ;
		\draw (vb) to (vn);
		
		\draw (va) to(vb) ;
		
		\draw[dotted] (1.7,1) -- (2.3,1);
		\end{tikzpicture}
		\caption{A triangular book graph}
		\label{figure: defensive allaince the traingluar book}
	\end{figure}
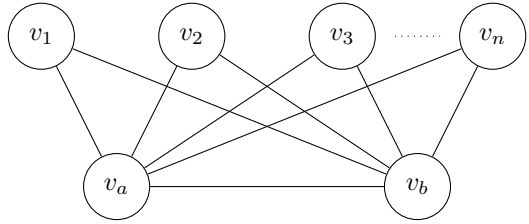
	
	The subsets of $V(G)$ with cardinality one which induce connected subgraphs in $G$, can be partitioned into the following parts: The part $\{ \{ v_{1} \},\{ v_{2} \}, \cdots , \{ v_{n} \} \}$ in which every set contributes the term $xy^{n}$ and by summing, we get the term $nxy^{n}$. The part $\{ \{ v_{a} \} ,\{ v_{b} \} \}$ in which every set contributes the term $xy$ and by summing, we get $2xy$.
	
	Now we prove the converse. Let $n$ be a positive integer, and $H$ is a graph with
	\begin{align*}
		[x]da(H;x,y) =& 2y + ny^{n} \text{.}
	\end{align*}
	
	By Proposition \ref{Proposition: Order}, the order of $H$ equals $n+2$. By Proposition \ref{Proposition: Degree sequence}, the degree sequence of $H$ is $(n+1,n+1,2,2, \cdots ,2)$.
	By connecting the two vertices with degree $n+1$ to every other vertex, $H$ is constructed which is isomorphic to the triangular book graph $B_{n}$.	
\end{proof}

\bigskip

\subsection{The quadrilateral book graph}\label{Subsection: quadrilateral book graph}
\begin{definition}
	Let $n$ be a positive integer. The \index{quadrilateral book graph}\emph{quadrilateral book graph} denoted by $B_{n,2}$ is defined by the graph join $nK_{2}+k_{2}$.
\end{definition}
\begin{proposition}
	A simple graph $G$ is isomorphic to the quadrilateral book graph $B_{n,2}$
	if and only if 
	\begin{align*}
		[x]da(G;x,y) =& 2y^{n+1} + 2ny^{2n} \text{ and}\\
		[x^{2}]da(G;x,y) =& ny^{2n+2} + (2n+1)y^{n+3} \text{ and} \\
		[x^{2n+1}]da(G;x,y) =& (2n+2)y^{2n+2} \text{ and} \\
		[x^{2n+2}]da(G;x,y) =& y^{2n+4} \text{, where $n$ is a positive integer.}
	\end{align*}
	
\end{proposition}
\begin{proof}
	First, we show that a simple graph $G$ which is isomorphic to a quadrilateral book graph $B_{n,2}$, has the above properties in the proposition. Let $G$ be of the form in Figure \ref{figure: defensive allaince the quadrilateral book}.
	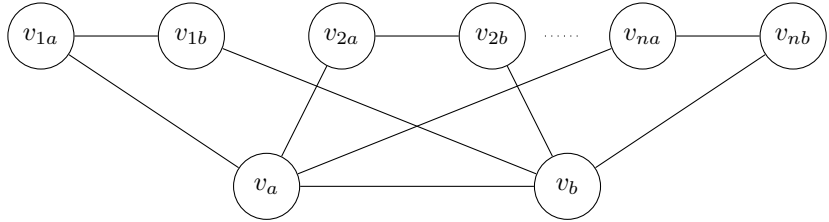
\begin{figure}[ht!]
		\centering
		\begin{tikzpicture}[baseline=(current bounding box.north)]
		\tikzset{vertex/.style = {shape=circle,draw,minimum size=2.5em}}
		
		\node[vertex] (va) at (-2,-1) {$v_{a}$};
		\node[vertex] (vb) at (2,-1) {$v_{b}$};
		\node[vertex] (v1) at  (-5,1) {$v_{1a}$};
		\node[vertex] (v2) at  (-3,1) {$v_{1b}$};
		\node[vertex] (v3) at  (-1,1) {$v_{2a}$};
		\node[vertex] (v4) at  (1,1) {$v_{2b}$};
		\node[vertex] (vn1) at  (3,1) {$v_{na}$};
		\node[vertex] (vn) at  (5,1) {$v_{nb}$};

		\draw (va)  to (v1);
		\draw (va) to (v3);
		\draw (va) to (vn1);
		\draw (vb) to (v2);
		\draw (vb) to (v4);
		\draw (vb) to (vn);
		
		\draw (v1) to (v2);
		\draw (v3) to (v4);
		\draw (vn1) to (vn);
		
		\draw (va) to(vb) ;
		
		\draw[dotted] (1.7,1) -- (2.2,1);
		\end{tikzpicture}
		\caption{A quadrilateral book graph}
		\label{figure: defensive allaince the quadrilateral book}
	\end{figure}
	
	The subsets of $V(G)$ with cardinality one which induce connected subgraphs in $G$, can be partitioned into the following parts: The part $\{ \{ v_{1a} \},\{ v_{1b} \},\{ v_{2a} \},\{ v_{2b} \},\\ \cdots , \{ v_{na} \},\{ v_{nb} \} \}$ in which every set contributes the term $xy^{2n}$ and by summing, we get the term $2nxy^{2n}$. The part $\{ \{ v_{a} \} ,\{ v_{b} \} \}$ in which every set contributes the term $xy^{n+1}$ and by summing, we get $2xy^{n+1}$.
	
	The subsets of $V(G)$ with cardinality two which induce connected subgraphs in $G$, can be partitioned into the following parts: The part $\{ \{ v_{1a} , v_{1b} \},\{ v_{2a} , v_{2b} \}, \cdots , \\ \{ v_{na} , v_{nb} \} \}$ in which every set contributes the term $x^{2}y^{2n+2}$ and by summing, we get the term $nx^{2}y^{2n+2}$. The part 
	$
	\{ \{ v_{a} , v_{b} \} , \{ v_{a} , v_{1a} \} , \{ v_{a} , v_{2a} \} , \cdots ,  \{ v_{a} , v_{na} \}
	, \{ v_{b} , v_{1b} \} , \{ v_{b} , v_{2b} \},\\ \cdots ,  \{ v_{b} , v_{nb} \}\}
	$
	in which every set contributes the term $x^{2}y^{n+3}$ and by summing, we get $(2n+1)x^{2}y^{n+3}$,
	
	The set $V(G)$ contributes the term $x^{2n+2}y^{2n+4}$. And if we delete any vertex from $V(G)$, we get a set which contributes the term $x^{2n+1}y^{2n+2}$ and by summing, we get $(2n+2)x^{2n+1}y^{2n+2}$.
	
	Now we prove the converse. Let $n$ be a positive integer, and $H$ is a graph with
	\begin{align*}
		[x]da(H;x,y) =& 2y^{n+1} + 2ny^{2n} \text{ and}\\
		[x^{2}]da(H;x,y) =& ny^{2n+2} + (2n+1)y^{n+3} \text{ and} \\
		[x^{2n+1}]da(H;x,y) =& (2n+2)y^{2n+2} \text{ and} \\
		[x^{2n+2}]da(H;x,y) =& y^{2n+4} \text{ .} 
	\end{align*}
	
	By Proposition \ref{Proposition: Order}, the order of $H$ equals $2n+2$. By Proposition \ref{Proposition: Size}, the size of $H$ equals $3n+1$. By Proposition \ref{Proposition: Connected}, $H$ is connected. By Proposition \ref{Proposition: Degree sequence}, the degree sequence of $H$ is $(n+1,n+1,2,2, \cdots ,2)$. By Proposition \ref{Proposition: number of cut vertices}, the number of cut vertices is zero. Let the two vertices with degree $n+1$ be $v_{a}$ and $v_{b}$ respectively. A subset of cardinality two which induces a connected subgraph in $G$, and contains two vertices of degree two is the only subset of cardinality two which contributes a term $[x^{2}y^{2n+2}]$. Then, the number of edges connecting two vertices of degree two is $[x^{2}y^{2n+2}]da(G;x,y)$ and equals $n$. The number of the rest edges is $2n+1$. But the number of edges which are incident to vertices of degree two are necessary only $2n$. Hence, the last edge is necessarily between the two vertices of degree $n+1$. At this point we have a graph like the one in Figure \ref{figure: defensive allaince the quadrilateral book constructing}
	\begin{figure}[ht!]
		\centering
		\begin{tikzpicture}[baseline=(current bounding box.north)]
		\tikzset{vertex/.style = {shape=circle,draw,minimum size=2.5em}}
		
		\node[vertex] (va) at (-2,-1) {$n+1$};
		\node[vertex] (vb) at (2,-1) {$n+1$};
		\node[vertex] (v1) at  (-5,1) {$2$};
		\node[vertex] (v2) at  (-3,1) {$2$};
		\node[vertex] (v3) at  (-1,1) {$2$};
		\node[vertex] (v4) at  (1,1) {$2$};
		\node[vertex] (vn1) at  (3,1) {$2$};
		\node[vertex] (vn) at  (5,1) {$2$};

		\draw (v1) to (v2);
		\draw (v3) to (v4);
		\draw (vn1) to (vn);
		
		\draw (va) to(vb) ;
		
		\draw[dotted] (1.7,1) -- (2.2,1);
		\end{tikzpicture}
		\caption{A quadrilateral book graph}
		\label{figure: defensive allaince the quadrilateral book constructing}
	\end{figure}
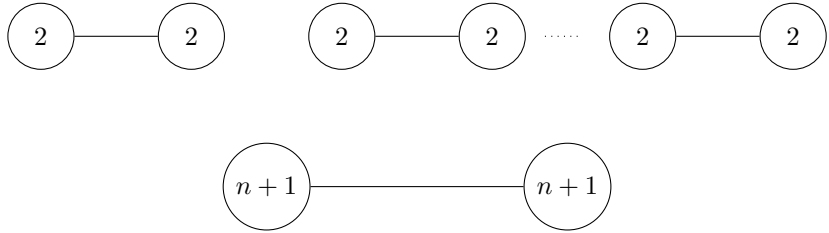
	where the number in the vertices is its degree. The two vertices of degree $n+1$ need to be connected to $n$ vertices of degree two. But a vertex with degree $n+1$ will never be connected to two adjacent vertices of degree two, since this will make this vertex of degree $n+1$ a cut vertex which contradicts the statement that $H$ has no cut vertices. This means that every vertex of degree $n+1$ will be connected to only non-adjacent vertices of degree two, which yields the quadrilateral book graph $H$.
\end{proof}

\section{Attaching a vertex to a complete graph}\label{Section: Attaching a vertex to a complete graph}
\begin{proposition}
	Let $v_{0}$ be a vertex and $n$ a positive integer. Let $H$ be a simple graph formed from $K_{n} \cup \{v_{0}\}$ by joining some vertices to $v_{0}$. Let $V(H) \setminus \{v_{0}\}= R \cup S$ where $R=\{r_{1},r_{2}, \cdots , r_{r}\}$, $r= |R|$ where $R$ is the set of vertices in $H$ which are adjacent to $v_{0}$ and $S=\{s_{1},s_{2}, \cdots , s_{s}\}$, $s= |S|$ where $S$ is the set of vertices in $H$ which are not adjacent to $v_{0}$. Let $G$ be a simple graph. $G$ is isomorphic to $H$ if and only if  
	\begin{align*}
		da(G;x,y)= &(1+xy^{2})da(K_{r};x,y) + y \; da(K_{s};x,y) + y \; da(K_{r};x,y)da(K_{s};x,y)\\
		&+ xy^{n+1-r} + xy \; da(K_{r};x,y)\sum_{j=1}^{s}\binom{s}{j}x^{j}y^{\min\{2j,s+1\}}  \text{ .}
	\end{align*}
	
\end{proposition}
\begin{proof}
	The subsets of $V(G)$ with cardinality one which induce connected subgraphs in $G$, can be partitioned into the following parts: 
	The part $\{ \{ v_{0}\} \}$ in which $\{ v_{0}\}$ contributes the term $xy^{n+1-r}$. \\
	The part containing the sets of cardinality $i$ in the range of $1 \leq i \leq r$ formed only from vertices in $R$ in which each set contributes the term $x^{i}y^{2i -1}$ and by summing, we get:
	\begin{align*}
		&\binom{r}{1}x^{1}y^{1} +\binom{r}{2}x^{2}y^{3} + \cdots + \binom{r}{r}x^{r}y^{2r-1} \\
		&=\sum_{i=1}^{r} \binom{r}{i}x^{i}y^{2i-1}  \\
		&=\frac{1}{y} \left( \sum_{i=0}^{r} \binom{r}{i}\left(xy^{2}\right)^{i} - 1 \right) \\
		&= \frac{\left(1+xy^{2}\right)^{r}-1}{y} \\
		&= da(K_{r};x,y)  \text{.}
	\end{align*}
	
	The part containing the sets of cardinality $i$ in the range of $1 \leq i \leq s$ arises only from the vertices in $S$ in which each set contributes the term $x^{i}y^{2i}$ and by summing, we get:
	\begin{align*}
		&\binom{s}{1}x^{1}y^{2} +\binom{s}{2}x^{2}y^{4} + \cdots + \binom{s}{s}x^{s}y^{2s} \\
		&=\sum_{i=1}^{s} \binom{s}{i}x^{i}y^{2i}  \\
		&=y \; \frac{1}{y} \left( \sum_{i=0}^{s} \binom{s}{i}\left(xy^{2}\right)^{i} - 1 \right) \\
		&= y \; \frac{\left(1+xy^{2}\right)^{s}-1}{y} \\
		&= y \; da(K_{s};x,y)  \text{.}
	\end{align*}
	
	The part containing the sets of cardinality $i$ in the range of $2 \leq i \leq r+1$ results from $\{ v_{0} \}$ and the vertices in $R$ in which each set contributes the term $x^{i+1}y^{2i+1}$ and by summing, we get:
	\begin{align*}
		&\binom{r}{1}x^{2}y^{3} +\binom{r}{2}x^{3}y^{5} + \cdots + \binom{r}{r}x^{r+1}y^{2r+1} \\
		&=\sum_{i=1}^{r} \binom{r}{i}x^{i+1}y^{2i+1}  \\
		&=xy^{2}\frac{1}{y} \left( \sum_{i=0}^{r} \binom{r}{i}\left(xy^{2}\right)^{i} - 1 \right) \\
		&= xy^{2}\frac{\left(1+xy^{2}\right)^{r}-1}{y} \\
		&= xy^{2}da(K_{r};x,y)  \text{.}
	\end{align*}
	
	The part containing the sets formed from subsets of $R$ of cardinality $i$ in the range of $1 \leq i \leq r$ and subsets of $S$ of cardinality $j$ in the range of $1 \leq j \leq s$ in which each set contributes the term $x^{i+j}y^{(r+s+1)+(i+j-1)-(r+s+1-i-j)}$ and by summing, we get:
	\begin{align*}
		&y\binom{r}{1}x^{1}y^{1} \binom{s}{1}x^{1}y^{1} +y\binom{r}{1}x^{1}y^{1} \binom{s}{2}x^{2}y^{3} + \cdots + y\binom{r}{1}x^{1}y^{1} \binom{s}{3}x^{3}y^{5}\\
		& +y\binom{r}{2}x^{2}y^{3} \binom{s}{1}x^{1}y^{1} +y\binom{r}{2}x^{2}y^{3} \binom{s}{2}x^{2}y^{3} + \cdots + y\binom{r}{2}x^{2}y^{3} \binom{s}{3}x^{3}y^{5}\\ 
		& \vdots\\
		&      +y\binom{r}{r}x^{r}y^{2r-1} \binom{s}{1}x^{1}y^{1} +y\binom{r}{r}x^{r}y^{2r-1} \binom{s}{2}x^{2}y^{3} + \cdots \\
		&+ y\binom{r}{r}x^{r}y^{2r-1} \binom{s}{s}x^{s}y^{2s-1}\\
		&=y\sum_{i=1}^{r} \binom{r}{i}x^{i}y^{2i-1} \sum_{j=1}^{s} \binom{s}{j}x^{j}y^{2j-1}  \\
		&=y \; \frac{1}{y} \left( \sum_{i=0}^{r} \binom{r}{i}\left(xy^{2}\right)^{i} - 1 \right) \; \frac{1}{y} \left( \sum_{j=0}^{s} \binom{s}{j}\left(xy^{2}\right)^{j} - 1 \right) \\
		&= y \; \frac{\left(1+xy^{2}\right)^{r}-1}{y} \; \frac{\left(1+xy^{2}\right)^{s}-1}{y}\\
		&= y \; da(K_{r};x,y) \; da(K_{s};x,y) \text{.}
	\end{align*}
	
	The part containing the sets formed from $v_0$ and subsets of $R$ of cardinality $i$ in the range of $1 \leq i \leq r$ and subsets of $S$ of cardinality $j$ in the range of $1 \leq j \leq s$ in which each set contributes the term $x^{i+j+1}y^{2i+min\{2j,s+1\}}$ and by summing, we get:
	\begin{align*}
		&xy\binom{r}{1}x^{1}y^{1} \binom{s}{1}x^{1}y^{\min\{2,s+1\}} +xy\binom{r}{1}x^{1}y^{1} \binom{s}{2}x^{2}y^{\min\{4,s+1\}} + \cdots \\
		&+ xy\binom{r}{1}x^{1}y^{1} \binom{s}{s}x^{s}y^{\min\{2s,s+1\}}\\
		&+ xy\binom{r}{2}x^{2}y^{3} \binom{s}{1}x^{1}y^{\min\{2,s+1\}} +xy\binom{r}{2}x^{2}y^{3} \binom{s}{2}x^{2}y^{\min\{4,s+1\}} + \cdots \\
		&+ xy\binom{r}{2}x^{2}y^{3} \binom{s}{s}x^{s}y^{\min\{2s,s+1\}}\\ 
		& \vdots\\
		&+ xy\binom{r}{r}x^{r}y^{2r-1} \binom{s}{1}x^{1}y^{\min\{2,s+1\}} + xy\binom{r}{r}x^{r}y^{2r-1} \binom{s}{2}x^{2}y^{\min\{4,s+1\}} + \cdots \\
		&+ xy\binom{r}{r}x^{r}y^{2r-1} \binom{s}{s}x^{s}y^{\min\{2s,s+1\}}\\
		&=xy\sum_{i=1}^{r} \binom{r}{i}x^{i}y^{2i-1} \sum_{j=1}^{s} \binom{s}{j}x^{j}y^{\min\{2j,s+1\}}  \\
		&=xy \; \frac{1}{y} \left( \sum_{i=0}^{r} \binom{r}{i}\left(xy^{2}\right)^{i} - 1 \right) \; \sum_{j=1}^{s} \binom{s}{j}x^{j}y^{\min\{2j,s+1\}} \\
		&= xy \; \frac{\left(1+xy^{2}\right)^{r}-1}{y} \; \sum_{j=1}^{s} \binom{s}{j}x^{j}y^{\min\{2j,s+1\}}\\
		&= xy\; da(K_{r};x,y)\; \sum_{j=1}^{s}\binom{s}{j}x^{j}y^{\min\{2j,s+1\}} \text{.}
	\end{align*}

	Now we prove the converse. Let $r$ and $s$ be integers and $H$ is a graph with the defensive alliance polynomial, 
	\begin{align*}
		da(H;x,y)=(1+xy^{2})da(K_{r};x,y) + y \; da(K_{s};x,y) + y \; da(K_{r};x,y)da(K_{s};x,y) \\
		+xy^{n+1-r} + xy \; da(K_{r};x,y) \; \sum_{j=1}^{s}\binom{s}{j}x^{j}y^{\min\{2j,s+1\}} \text{.}
	\end{align*}
	By Proposition \ref{Proposition: Order}, the order of $H$ equals $r+s+1$. Let $r+s=n$.
	By Proposition \ref{Proposition: Degree sequence}, the degree sequence of $H$ consist of $n$ $r$ times then ($n-1$) $s$ times then r one time: $(n,n, \cdots ,n,n-1,n-1 \cdots, n-1,r)$.
	By constructing first all the vertices with degree $n$.	Note that no term left of the form $ax^{2}y^{(r+s+1)+1-(r-1)}$, hence no vertex with degree $n-1$ is connected to the vertex of degree $r$. Hence we choose arbitrary $s$ vertices and connect them to each other.
\end{proof}

\section{The distinctive power of the defensive alliance polynomial} \label{Section: Distinctive power of $DA$}
The authors in \cite{Carb14}, showed how the alliance polynomial can characterize some classes of graphs which were not characterized by other well-known graph polynomials like the tutte polynomial, the domination polynomial, the independence polynomial, the matching polynomial, the bivariate polynomial, and the subgraph component polynomial.

As a generalization for the alliance polynomial, the defensive alliance polynomial has at least the same power. In this section, we present two pairs of graphs that cannot be characterized by the alliance polynomial but can be characterized by the defensive alliance polynomial.\\

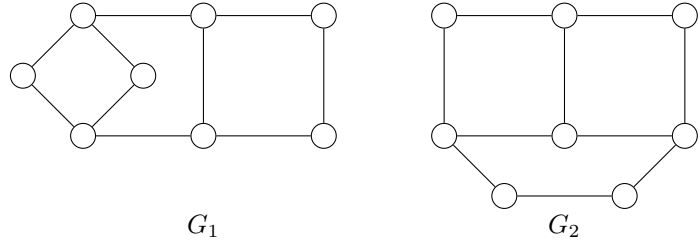
\begin{figure}[h]
	\centering
	\begin{tikzpicture}
	\begin{scope}[{scale=0.8}]
	\tikzset{vertex/.style = {shape=circle}}
	
	\node[vertex,draw] (v1) at  (0,1) 	{};
	\node[vertex,draw] (v2) at (0,-1) 	{};
	\node[vertex,draw] (v3) at (2,1) 	{};
	\node[vertex,draw] (v4) at  (2,-1) 	{};
	\node[vertex,draw] (v5) at  (-2,1) 	{};
	\node[vertex,draw] (v6) at  (-2,-1) 	{};
	\node[vertex,draw] (v7) at  (-1,0) 	{};
	\node[vertex,draw] (v8) at  (-3,0) 	{};
	\node[vertex] 	(t) at (0,-2.5) {$G_{1}$};
	
	\draw (v1) -- (v2) -- (v4) -- (v3) -- (v1) --(v5) --(v7) --(v6) --(v8) --(v5) 
	(v6) --(v2);

	\begin{scope}[xshift=6cm] 	
	\node[vertex,draw] (v1) at  (0,1) 	{};
	\node[vertex,draw] (v2) at (0,-1) 	{};
	\node[vertex,draw] (v3) at (2,1) 	{};
	\node[vertex,draw] (v4) at  (2,-1) 	{};
	\node[vertex,draw] (v5) at  (-2,1) 	{};
	\node[vertex,draw] (v6) at  (-2,-1) 	{};
	\node[vertex,draw] (v7) at  (-1,-2) 	{};
	\node[vertex,draw] (v8) at  (1,-2) 	{};
	\node[vertex] 	(t) at (0,-2.5) {$G_{2}$};
	
	\draw (v1) -- (v2) -- (v4) -- (v3) -- (v1) --(v5) --(v6) --(v2)  
	(v6) --(v7) --(v8)--(v4);
	\end{scope}	
	\end{scope}
	\end{tikzpicture}
	\caption{First pair of graphs}
	\label{figure: defensive allaince distincitve power}
\end{figure}
The alliance polynomial of the two graphs in the Figure \ref{figure: defensive allaince distincitve power} is:
\[
A(G_{1};x) = A(G_{2};x) = x^{10} + 7x^{9} + 37x^{8} + 63x^{7} + 4x^{6} + 4x^{5} \text{.}
\]
The defensive alliance polynomial of $G_{1}$:
\begin{align*}
da(G_{1};x,y) =& x^{8}y^{10} + 2x^{7}y^{9} + 6x^{7}y^{8} + x^{6}y^{9} + 14x^{6}y^{8} + 7x^{6}y^{7}\\
& + 2x^{5}y^{9} + 10x^{5}y^{8} + 16x^{5}y^{7} + 2x^{4}y^{9} + 4x^{4}y^{8}+ 17x^{4}y^{7}\\
& + 2x^{3}y^{8} + 14x^{3}y^{7} + x^{2}y^{8} + 9x^{2}y^{7} + 4xy^{6} + 4xy^{5}\text{.}
\end{align*}
The defensive alliance polynomial of $G_{2}$:
\begin{align*}
da(G_{2};x,y) =& x^{8}y^{10} + 3x^{7}y^{9} + 5x^{7}y^{8} + x^{6}y^{9} + 15x^{6}y^{8} + 7x^{6}y^{7} \\
&+ x^{5}y^{9} + 11x^{5}y^{8}+ 15x^{5}y^{7} + 2x^{4}y^{9} + 2x^{4}y^{8} + 19x^{4}y^{7}\\
& + 3x^{3}y^{8} + 13x^{3}y^{7} + x^{2}y^{8} + 9x^{2}y^{7} + 4xy^{6} + 4xy^{5}\text{.}
\end{align*}
Another pair of graphs:\\
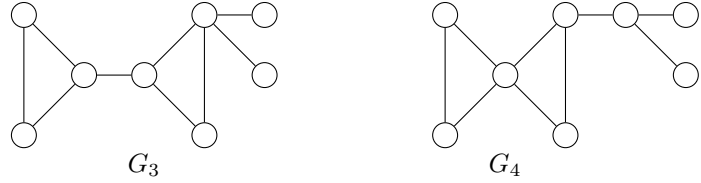
\begin{figure}[h]
	\centering
	\begin{tikzpicture}
	\begin{scope}[{scale=0.8}]
	\tikzset{vertex/.style = {shape=circle}}
	
	\node[vertex,draw] (v1) at  (0,0) 	{};
	\node[vertex,draw] (v2) at (1,1) 	{};
	\node[vertex,draw] (v3) at (1,-1) 	{};
	\node[vertex,draw] (v5) at  (2,1) 	{};
	\node[vertex,draw] (v6) at  (2,0) 	{};
	\node[vertex,draw] (v7) at  (-1,0) 	{};
	\node[vertex,draw] (v8) at  (-2,1) 	{};
	\node[vertex,draw] (v9) at  (-2,-1) 	{};
	\node[vertex] 	(t) at (0,-1.5) {$G_{3}$};
	
	\draw (v1) -- (v2) -- (v3) -- (v1)  
	(v2) --(v5)
	(v2) --(v6)
	(v1) --(v7) --(v8) --(v9) -- (v7);

	\begin{scope}[xshift=6cm] 	
	\node[vertex,draw] (v1) at  (0,0) 	{};
	\node[vertex,draw] (v2) at (1,1) 	{};
	\node[vertex,draw] (v3) at (1,-1) 	{};
	\node[vertex,draw] (v5) at  (3,1) 	{};
	\node[vertex,draw] (v6) at  (3,0) 	{};
	\node[vertex,draw] (v7) at  (2,1) 	{};
	\node[vertex,draw] (v8) at  (-1,1) 	{};
	\node[vertex,draw] (v9) at  (-1,-1) 	{};
	\node[vertex] 	(t) at (0,-1.5) {$G_{4}$};
	
	\draw (v1) -- (v2) -- (v3) -- (v1)  
	(v7) --(v5)
	(v7) --(v6)
	(v2) -- (v7)
	(v1) --(v8) --(v9) -- (v1);
	\end{scope}	
	\end{scope}
	\end{tikzpicture}
	\caption{Second pair of graphs}
	\label{figure: defensive allaince distincitve power2}
\end{figure}
The alliance polynomial of the two graphs in the Figure \ref{figure: defensive allaince distincitve power2} is:
\[
A(G_{3};x) = A(G_{4};x) = 8x^{9} + 26x^{8} + 20x^{7} + 11x^{6} + 2x^{5} + x^{4} \text{.}
\]
The defensive alliance polynomial of $G_{3}$ is:
\begin{align*}
da(G_{3};x,y) =& x^{8}y^{9} + 3x^{7}y^{9} + 2x^{7}y^{8} + 9x^{6}y^{8} + x^{6}y^{7} + x^{5}y^{9}\\
& + 7x^{5}y^{8} + 3x^{5}y^{7} + x^{5}y^{6} + 2x^{4}y^{9} + 3x^{4}y^{8}+ 5x^{4}y^{7}\\
& + 2x^{4}y^{6} + x^{3}y^{9} + 4x^{3}y^{8} + 5x^{3}y^{7} + x^{3}y^{6} + x^{2}y^{8}\\
&+ 4x^{2}y^{7} + 4x^{2}y^{6} + 2xy^{7} + 3xy^{6} + 2xy^{5} + xy^{4}
\text{.}
\end{align*}
The defensive alliance polynomial of $G_{4}$ is:
\begin{align*}
da(G_{4};x,y) =& x^{8}y^{9} + 3x^{7}y^{9} + 2x^{7}y^{8} + 2x^{6}y^{9} + 7x^{6}y^{8} + x^{6}y^{7}\\
& + x^{5}y^{9} + 7x^{5}y^{8}+ 3x^{5}y^{7} + x^{5}y^{6} + 5x^{4}y^{8} + 5x^{4}y^{7}\\
& + 2x^{4}y^{6} + x^{3}y^{9} + 4x^{3}y^{8} + 5x^{3}y^{7} + x^{3}y^{6} + x^{2}y^{8}\\
&+ 4x^{2}y^{7}+ 4x^{2}y^{6} + 2xy^{7} + 3xy^{6} + 2xy^{5} + xy^{4}
\text{.}
\end{align*}

\end{document}